\documentclass[12pt]{amsart}
\calclayout

\usepackage[lite]{amsrefs}

\usepackage{amssymb}
\usepackage{amsbsy}

\usepackage[all,cmtip]{xy}

\usepackage{graphicx}

\newcommand{\Line}{\ell}

\newcommand{\LI}{\Line_{\infty}}
\newcommand{\LIP}{\Line_{\infty,P}}
\newcommand{\LIS}{\Line_{\infty,S}}
\newcommand{\LIV}{\Line_{\infty,V}}

\newcommand{\bbR}{\mathbb{R}}

\newcommand{\bbP}{\mathbb{P}}

\newcommand{\descR}{\bbR\bbP}

\newcommand{\descRt}{\bbR\bbP^2}

\newcommand{\descP}{\descR_P}

\newcommand{\descPt}{\descRt_P}
\newcommand{\descSt}{\descRt_S}
\newcommand{\descVt}{\descRt_V}

\newcommand{\descShortP}{P}
\newcommand{\descShortS}{S}
\newcommand{\descShortV}{V}

\newcommand{\iso}{\mathcal{I}}
\newcommand{\isoEq}{\cong}

\newcommand{\isoPS}{\iso_{\descShortP\descShortS}}

\newcommand{\isoPV}{\iso_{\descShortP\descShortV}}

\newcommand{\Span}{\mathrm{span}}

\usepackage{hyperref} \hypersetup{ colorlinks=false, linkcolor=red, filecolor=magenta, urlcolor=blue, citecolor=green, pdftitle={Overleaf Example}, pdfpagemode=FullScreen, }

\def\nvec{\boldsymbol{n}}

\numberwithin{equation}{section}

\theoremstyle{plain} %% This is the default, anyway
\newtheorem{theorem}[equation]{Theorem}

\newtheorem{lemma}[equation]{Lemma}

\theoremstyle{definition}
\newtheorem{definition}[equation]{Definition}

\theoremstyle{remark}
\newtheorem{remark}[equation]{Remark}

\begin{document}

%%% In the title, use a double backslash "\\" to show a linebreak:
%%% Use one of the following two forms:
%%% \title{Text of the title}
%%% or
%%% \title[Short form for the running head]{Text of the title}
\title{On Real Projective Plane Constructions and Their Isomorphisms}

%%% If there are multiple authors, they're described one at a time:
%%% First author: \author{} \address{} \curraddr{} \email{} \thanks{}
%%% Second author: \author{} \address{} \curraddr{} \email{} \thanks{}
%%% Third author: \author{} \address{} \curraddr{} \email{} \thanks{}
\author{Noah Everett} \curraddr{Department of Mathematics, South Dakota School of Mines and Technology, Rapid City, SD 57701} \curraddr{Department of Physics, South Dakota School of Mines and Technology, Rapid City, SD 57701} \email{Noah.Everett@mines.sdsmt.edu}
\author{Patrick Fleming} \curraddr{Department of Mathematics, South Dakota School of Mines and Technology, Rapid City, SD 57701} \email{Patrick.Fleming@sdsmt.edu}

%%% In the address, show linebreaks with double backslashes:
\address{}

%%% Current address is optional.
% \curraddr{}

%%% Email address is optional.
% \email{}

%%% If there's a second author:
% \author{}
% \address{}
% \curraddr{}
% \email{}

%%% To have the current date inserted, use \date{\today}:
% \date{\today}

\maketitle

%%% To include an abstract, uncomment the following two lines and type
%%% the abstract in between them:
\begin{abstract}
The real projective plane ($\descR^2$) has three well known isomorphic constructions: the extended Euclidean plane, unit (hemi)sphere, and $\bbR^3$ vector space.
In this paper, we find isomorphisms that map between these three constructions.
Additionally, we investigate their relationship to direction-sensitive photosensors which use lens(es) to transform light's direction to a position on a local plane.
This transformation, done by lenses, is a physical version of an isomorphism between projective plane constructions.
\end{abstract}

%%% To include a table of contents, uncomment the following line:
% \tableofcontents

%%%-------------------------------------------------------------------
%%%-------------------------------------------------------------------
%%% Start the body of the paper here!  E.G., maybe use:
%%% \section{Introduction}
%%% \label{sec:intro}

%%% For a numbered display, use
%%% \begin{equation}
%%%   \label{something}
%%%   The display goes here
%%% \end{equation}
%%% and you can refer to it as \eqref{something}.

%%% For an unnumbered display, use
%%% \begin{equation*}
%%%   The display goes here
%%% \end{equation*}

%%% To import a graphics file, you must have said
%%% \usepackage{graphicx}
%%% in the preamble (i.e., before the \begin{document}).
%%% Putting it into a figure environment enables it to float to the
%%% next page if there isn't enough room for it on the current page.
%%% The \label command must come after the \caption command.
% \begin{figure}[h]
%   \includegraphics{filename}
%   \caption{Some caption}
%   \label{somelabel}
% \end{figure}

\begin{figure}[h]
    \centering
    \includegraphics[width=1.0\linewidth]{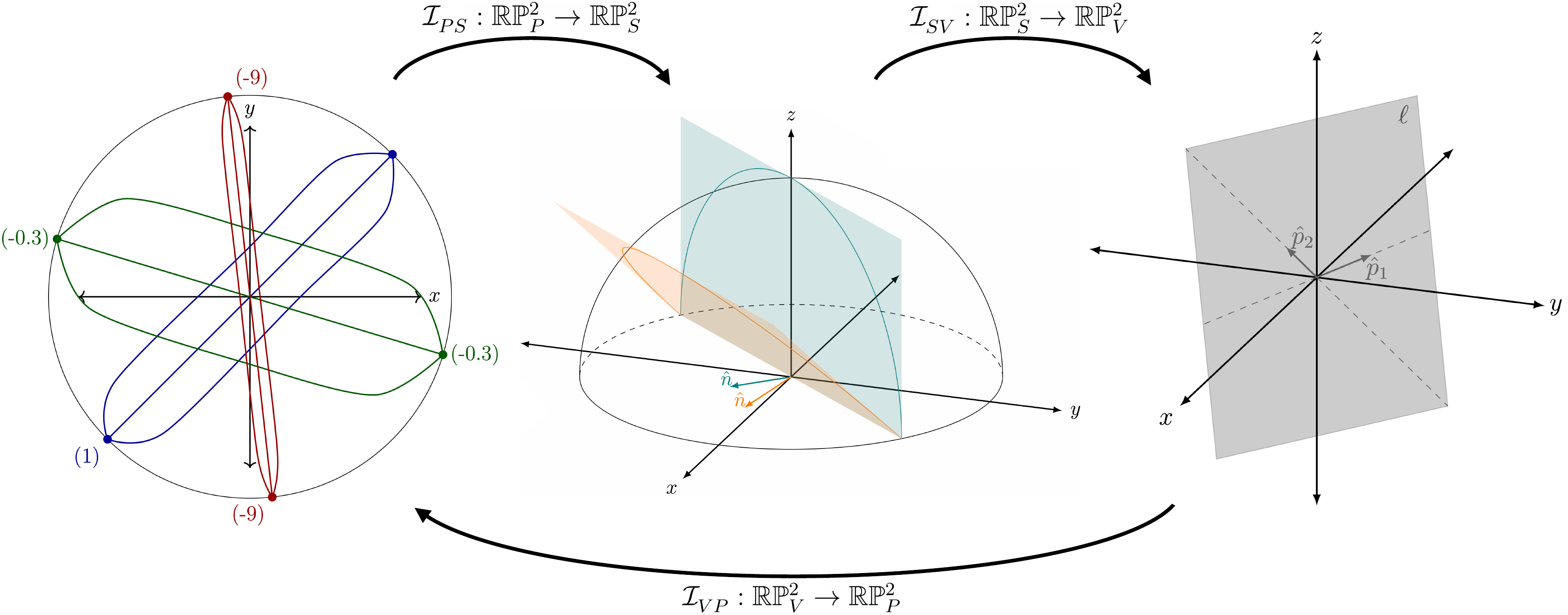}
    \caption{
    Visualization of real projective plane constructions and their isomorphisms.
    % An arbitrary set of points and lines is shown in each construction.
    \textbf{Left:} Extended Euclidean plane with three chromatically separated groups of lines each with different, arbitrarily chosen slopes.
    The point that each line is incident with the line at infinity is notated by the slope of the line.
    \textbf{Center:} (hemi)sphere construction with two semicircles (lines), the plane containing them, and a normal vector to that plane. 
    \textbf{Right:} $\bbR^3$ vector space with two 1-dimensional subspaces (points) and the plane (line) incident with them. 
    }
    \label{fig:Constructions}
\end{figure}

%%%%%%%%%%%%%%%%%%%%%%%%%%%%%%%%%%%%%%%%%%%%%%%%%%%%%%%%%%%%%%%%%%%%%%%%%%%
%%%%%%%%%%%%%%%%%%%%%%%%%%%%%%% Section %%%%%%%%%%%%%%%%%%%%%%%%%%%%%%%%%%%
%%%%%%%%%%%%%%%%%%%%%%%%%%%%%%%%%%%%%%%%%%%%%%%%%%%%%%%%%%%%%%%%%%%%%%%%%%%
% \section{Direction-Sensitive photosensors}\label{sec:DSPD}
\section{Motivation}\label{sec:DSPD}
%%%%%%%%%%%%%%%%%%%%%%%%%%%%%%%%%%%%%%%%%%%%%%%%%%%%%%%%%%%%%%%%%%%%%%%%%%%
\begin{figure}[h]
    \centering
    % \begin{subfigure}{0.49\linewidth}
    %     \includegraphics[width=0.9\textwidth,valign=c]{Figures/DSPD.png}
    %     \label{fig:enter-label}
    % \end{subfigure}
    % \hfill
    % \begin{subfigure}{0.5\linewidth}
    %     \includegraphics[width=0.9\textwidth,valign=c]{Figures/G4_detector_v01._white.png}
    %     \label{fig:enter-label}
    % \end{subfigure}
    \includegraphics[width=\textwidth]{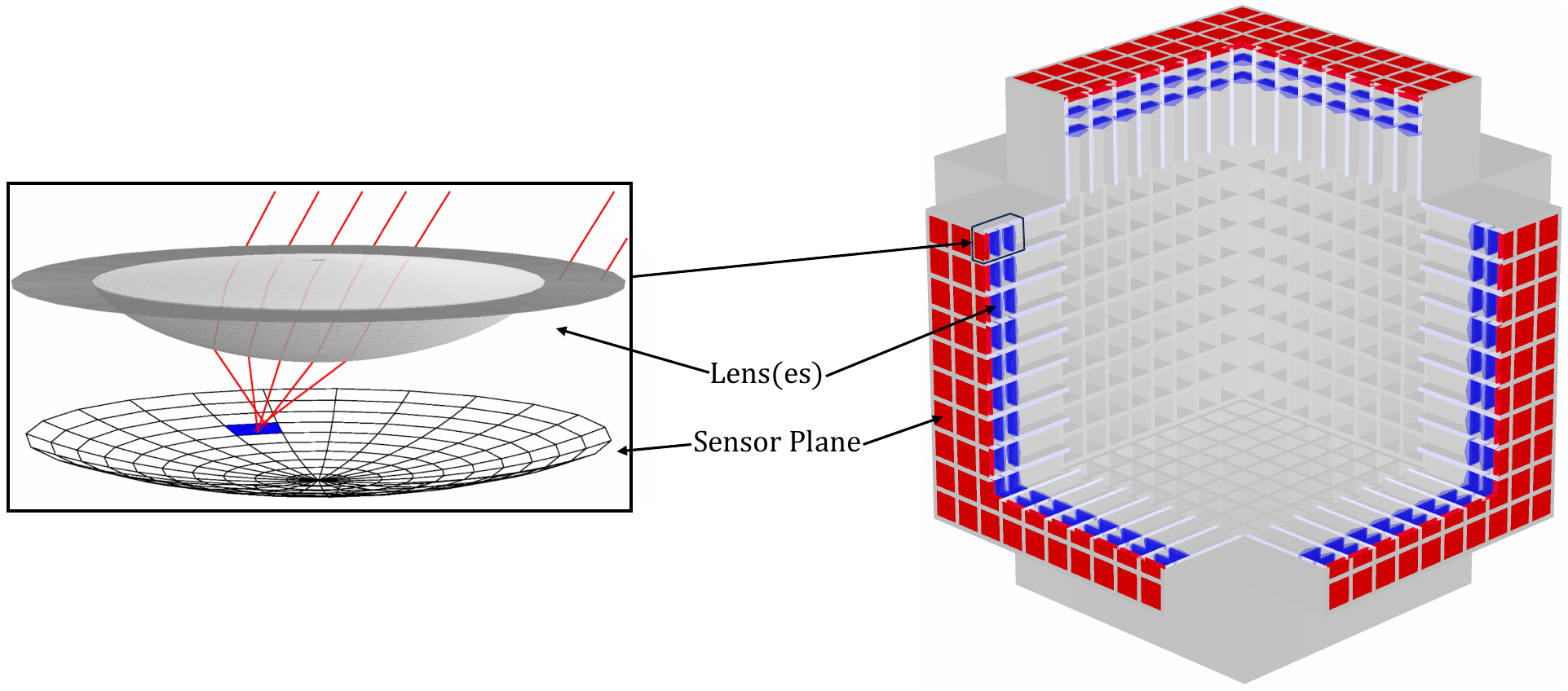}
    \caption{
    \textbf{Left:} 
    Conceptual design of a direction-sensitive photosensor where
    red lines represent photons. 
    These photons enter the sensor from above, are refracted by the intermediate lens, and are finally detected on the bottom sensor plane.
    Figure taken from Ref.~\cite{Dalmasson_2018}.
    \textbf{Right:}
    A larger detector outfitted with direction-sensitive photosensors.
    The sensor planes are red and the lenses are blue. 
    This detector would be filled with a scintillating material which would produce photons when excited by an incident particle.
    }
    \label{fig:DSPD}
\end{figure}

Many modern particle physics experiments use scintillating materials as detection mediums, i.e. events (particle interactions or decays) occur within the scintillator.
Scintillating materials emit light isotropically when excited by a charged particle with sufficient energy.
This light is then detected by photosensors which are used to determine event characteristics.
Commonly used scintillators include liquid or gaseous noble elements such as argon and xenon, or liquid hydrocarbons.
Scintillators are advantageous for studying low-energy events because of their improved energy sensitivity due to their increased light emittance compared to non-scintillating materials like water, where charged particles produce photons via the Cherenkov effect.

To infer the position of an event in a scintillation detector, experiments have traditionally relied on only information gathered by detecting photons produced by the scintillator, those being the time and location of detection.
of photons which are produced by the scintillator.
However, a novel optical detection technique, motivated by Ref.~\cite{Dalmasson_2018}, which we call direction-sensitive photosensors, also provides information about the direction of detected photons.
To do this, a lens system is used to redirect light to specific points on the sensor plane as seen in Figure~\ref{fig:DSPD}.
The final location of light on the sensor plane is correlated with its incoming direction, effectively transforming the light's direction to a physical location.
This technology offers the potential to improve the position resolution for scintillation-based detectors.

When investigating these sensors, we noticed their similarities to the constructions of projective planes.
One can think of the incoming light rays as vectors, with parallel light rays being parallel classes to the base vector (incident with the origin), and the sensor plane as the extended Euclidean plane.
With this description, the lens acts as an isomorphism from the vector space construction to the extended Euclidean plane construction, which are described later.
%%%%%%%%%%%%%%%%%%%%%%%%%%%%%%%%%%%%%%%%%%%%%%%%%%%%%%%%%%%%%%%%%%%%%%%%%%%
%%%%%%%%%%%%%%%%%%%%%%%%%%%%%%%%%%%%%%%%%%%%%%%%%%%%%%%%%%%%%%%%%%%%%%%%%%%
%%%%%%%%%%%%%%%%%%%%%%%%%%%%%%%%%%%%%%%%%%%%%%%%%%%%%%%%%%%%%%%%%%%%%%%%%%%

%%%%%%%%%%%%%%%%%%%%%%%%%%%%%%%%%%%%%%%%%%%%%%%%%%%%%%%%%%%%%%%%%%%%%%%%%%%
%%%%%%%%%%%%%%%%%%%%%%%%%%%%%%% Section %%%%%%%%%%%%%%%%%%%%%%%%%%%%%%%%%%%
%%%%%%%%%%%%%%%%%%%%%%%%%%%%%%%%%%%%%%%%%%%%%%%%%%%%%%%%%%%%%%%%%%%%%%%%%%%
\section{Introduction}
%%%%%%%%%%%%%%%%%%%%%%%%%%%%%%%%%%%%%%%%%%%%%%%%%%%%%%%%%%%%%%%%%%%%%%%%%%%
\begin{definition}
A \textit{projective plane} 
% $\descRt$ (or $PG(2,\bbR)$) 
is a space comprised of points 
% in the real space 
that satisfies the postulates of projective geometry:
\begin{enumerate}
    \item For any two distinct points, there exists a unique line incident with both of them.
    % \item Given any two distinct points, there is a unique line incident with both of them.
    \item For any two distinct lines, there is a unique point incident with both of them.
    % \item Given any two distinct lines, there is a unique point incident with both of them.
    \item There exist at least four unique points such that no line is incident with more than two of them.
    % \item There exist at least four unique points such that no line is incident with more than two of them.
\end{enumerate}
\end{definition}

% \begin{definition}
%     Let \textit{real projective plane} be the 2-dimensional real projective space $\descRt$.
% \end{definition}
    
% \begin{remark}\label{rem:dim}
%     Given a space $\mathcal{S}$ which contains a set $\mathcal{T}$ such that $\mathcal{T}\subseteq\mathcal{S}$, 
%     then a space $\mathcal{S}'\mid \mathcal{S}\subset\mathcal{S}'$ contains $\mathcal{T}'$ such that $\mathcal{T}'\subset\mathcal{S}'$
%     and $\mathcal{T}\cong\mathcal{T}'$.
% \end{remark}

\begin{definition}\label{def:Iso}
    An \textit{isomorphism} is a bijective function that maps points to points and preserves incidence structure.
    Let $\psi:A\to B$ be a bijective function that maps points in $A$ to points in $B$.
    For $\psi$ to be an isomorphism, it must preserve lines such that, for all lines $\ell_A\subset A$ where $\ell_A=\{a_1,a_2,\dots\}$, then $\{\psi(a_1),\psi(a_2),\dots\}=\ell_B$ which must be a line in $B$.
    If such a function exists, we write $A\isoEq B$.
\end{definition}

The real projective plane 
$\descRt$ (or $PG(2,\bbR)$) 
has several well-known isomorphic constructions: 
the extended Euclidean plane $\descPt$, 
the unit $\bbR^3$ (hemi)sphere $\descSt$, 
and the $\bbR^3$ vector space $\descVt$ such that 
\mbox{ $\descRt\isoEq\descPt\isoEq\descSt\isoEq\descVt$. }
In this paper, we find isomorphisms 
\mbox{$\iso_{AB}$} 
that map between constructions of $\descRt$ such that 
\mbox{$\iso_{AB}:\descR^2_A\to\descR^2_B$}
where 
\mbox{$A,B\in\{\descShortP,\descShortS,\descShortV\}$} 
and 
\mbox{$A\neq B$}.
Namely, we present isomorphisms from $\descPt$ to $\descSt$ and $\descVt$.

% Additionally, given two projective spaces
% \mbox{$\descR^a,\descR^b\mid a,b\in\bbN\backslash\{1\}$},
% if $a>b$ then
% \mbox{$\descR^b\subset\descR^a$}
% and, using Remark~\ref{rem:dim}, we know $\descR^a$ inherits constructions from $\descR^b$ such that
% $\descP^b,\descS^b,\descV^b\subset\descR^a$ 
% where 
% $\descPn\cong\descSn\cong\descVn\cong\descPt\cong\descSt\cong\descVt$.

\begin{definition}\label{def:PI}
    In the extended Euclidean plane description ($\descPt$) of $\descRt$, let the \textit{points at infinity} be the set $\{ (p)\mid p\in\bbR\cup\{\infty\} \}$.
    These points are an infinite distance (in terms of the standard norm) from the origin such that, for a point $(p)$ in Cartesian coordinates, $\lim_{x\to\infty}(x,px)=(p)$ where $p$ is the slope away from the origin with respect to the $y$-axis.
    Additionally, we call the point $(\infty)=\lim_{y\to\infty}(c,y)\mid c\in\bbR$ the \textit{point of infinity}.
\end{definition}

\begin{definition}\label{def:LI}
    The \textit{line at infinity} $\LIP$ in $\descPt$ contains all the points at infinity such that $\LIP=\{(p)\mid p\in\bbR\cup\{\infty\}\}$.
    $\LIP$ can also be defined as containing all of the points in $\descR^2$ that are not in the real space $\bbR^2$ such that
    \mbox{$\LIP =\descR^2\cap(\bbR^2)^c=\descR^2-\bbR^2$}.
\end{definition}

% \begin{definition}\label{def:PI} (Point at infinity). 
%     The point $(\infty)\in\LI$ is called the point at infinity.
% \end{definition}

\begin{definition}\label{def:PPn}
    % The extended Euclidean plane construction $\descPt\isoEq\descRt$ or, more generally, 
    The \textit{extended Euclidean plane construction} $\descPt$ of the real projective plane is the $\bbR^2$ space with the addition of the line at infinity, such that 
    \mbox{$\descPt=\bbR^2\cup\LIP$}.
    In this construction, a line $\ell_P\subset\descPt$ with slope $m$ and $y$-intercept $b$ is equivalent to the line $\ell_\bbR=\{(x,mx+b)\mid x\in\bbR\}\subset\bbR^2$ with the addition of the point at infinity $(m)$, such that $\ell_P=\ell_\bbR\cup\{(m)\}$.
    % and
    % \mbox{$\descPn\isoEq\descRn$}.
\end{definition}

\begin{theorem}
    The extended Euclidean plane construction $(\descPt)$ is a construction of the real projective plane.
\end{theorem}
\begin{proof}
    The extended Euclidean plane construction meets the axioms of projective geometry:
    \begin{enumerate}
    \item For any two distinct points $p_1$ and $p_2$, there exists a unique line $\ell\subset\descPt$ incident with both of them.
    \begin{itemize}
    \item When $p_1=(p_{1x},p_{1y}), p_2=(p_{2x},p_{2y}) \in\bbR^2$ and $p_{1x}\neq p_{2x}$, the line $\ell$ can simply be written in point slope form such that 
    \begin{equation*}
        \ell=\left\{\left(x,\frac{p_{2y}-p_{1y}}{p_{2x}-p_{1x}}(x-p_{1x})+y_{1x}\right)\mid x\in\bbR\right\}\cup\left\{\left(\frac{p_{2y}-p_{1y}}{p_{2x}-p_{1x}}\right)\right\}.
    \end{equation*}
    \item When $p_1=(p_{1x},p_{1y}), p_2=(p_{2x},p_{2y}) \in\bbR^2$ and $p_{1x}=p_{2x}$, the line $\ell$ is vertical such that $\ell=\left\{\left(p_{1x},y\right)\mid y\in\bbR\right\}\cup\{(\infty)\}$.
    \item When $p_1,p_2\in\LIP$, the line $\ell$ is the line at infinity such that $\ell=\LIP$.
    \item When $p_1=(p_{1x},p_{1y})\in\bbR^2$ and $p_{2}=(m)\in\LIP$ and $(m)\neq(\infty)$, the line $\ell$ has a slope $m$ such that $\ell=\left\{\left(x,m(x-p_{1x})+p_{1y}\right)\mid x\in\bbR\right\}\cup\{(m)\}$.
    \item When $p_1=(p_{1x},p_{1y})\in\bbR^2$ and $p_{2}=(\infty)$, the line $\ell$ is a vertical line such that $\ell=\left\{\left(p_{1x},y\right)\mid y\in\bbR\right\}\cup\{(\infty)\}$.
    \end{itemize}
    \item For any two distinct lines $\ell_1$ and $\ell_2$, there is a unique point $p$ incident with both of them.
    \begin{itemize}
    \item When $\ell_1$ and $\ell_2$ are non-vertical lines such that $\ell_1 = \{(x,m_1x+b_1)\mid x\in\bbR\}\cup\{(m_1)\}$ and $\ell_2 = \{(x,m_2x+b_2)\mid x\in\bbR\}\cup\{(m_2)\}$ where $m_1\neq m_2$, the point $p$ can simply be written as $p=\left(\frac{b_2-b_1}{m_1-m_2},\frac{m_1b_2-m_2b_1}{m_1-m_2}\right)$.
    \item When $\ell_1$ is a non-vertical line $\ell_1 = \{(x,m_1x+b_1)\mid x\in\bbR\}\cup\{(m_1)\}$ and $\ell_2$ is a vertical line $\ell_2 = \{(b_2,y)\mid y\in\bbR\}\cup\{(\infty)\}$, the point $p=\left(b_2,m_1b_2+b_1\right)$.
    \item When $\ell_1$ and $\ell_2$ have the same slope such that $\ell_1 = \{(x,mx+b_1)\mid x\in\bbR\}\cup\{(m)\}$ and $\ell_2 = \{(x,mx+b_2)\mid x\in\bbR\}\cup\{(m)\}$ where $b\neq b$, the point $p=\left(m\right)$.
    \item When $\ell_1$ is a non-vertical line $\ell_1 = \{(x,mx+b)\mid x\in\bbR\}\cup\{(m)\}$ and $\ell_2$ is the line at infinity $\ell_2 = \LIP$, the point $p=\left(m\right)$.
    \item When $\ell_1$ is a vertical line such that $\ell_1 = \{(b,y)\mid y\in\bbR\}\cup\{(\infty)\}$ and $\ell_2$ is the line at infinity $\ell_2 = \LIP$, the point $p=\left(\infty\right)$.
    \end{itemize}
    \item There exist at least four unique points such that no line is incident with more than two of them.
    \vspace{1mm}\\
    The points $(0,0)$, $(1,0)$, $(0,1)$, and $(1,1)$ satisfy this axiom.\qedhere
    \end{enumerate}
\end{proof}

% \begin{definition}\label{def:PPt}
%     Let the \textit{extended Euclidean plane construction} be the 2-dimensional $\descPn$, such that 
%     $\descPt\isoEq\descRt$.
%     It follows that $\descPt=\bbR^2\cup\LI^2$.
% \end{definition}

\begin{definition}\label{def:SPn}
    % The hemisphere construction $\descS$ of $\descR$ is isomorphic with respect to $\descR$, such that $\descP\simeq\descR$.
    Let the \textit{hemisphere construction} $\descSt$ of the real projective plane be exactly half of the points on the sphere
    \mbox{$S^{2}\subset\bbR^{3}$} such that $\descSt\subset S^2$.
    For the sake of generality, we define $S^2$, and thus $\descSt$, to have a radius of $\rho$. 
    $\rho$ is commonly 1 giving a unit hemisphere.
    In this work, for
    \mbox{$s=(s_1,s_2,s_3)\in\bbR^{3}$},
    we let 
    \begin{equation*}
    \LIS = \left\{ s \mid  \sqrt{s_1^2 + s_2^2 + s_3^2}=\rho, s_3=0,-\pi/2<\tan^{-1}s_2/s_1\leq\pi/2 \right\}
    \end{equation*}
    and $H=\left\{ s \mid  \sqrt{s_1^2 + s_2^2 + s_3^2} = \rho,s_3>0 \right\}$.
    Then, $\descSt=H\cup\LIS$.
    We note that this is not the only possible definition for $\LIS$ or $\descSt$ itself.
    In $\descSt$, lines are great semicircles on the surface of $S^2$.
    This construction is generally used as a unit sphere where antipodal points are identified together, and thus its lines would be entire great circles.
\end{definition}

\begin{theorem}
    The hemisphere construction $(\descSt)$ is a construction of the real projective plane.
\end{theorem}
\begin{proof}
    The hemisphere construction meets the axioms of projective geometry:
    \begin{enumerate}
    \item For any two distinct points $p_1=(p_{1\rho},p_{1\theta},p_{1\phi})$ and $p_1=(p_{2\rho},p_{2\theta},p_{2\phi})$, there exists a unique line $\ell\subset\descSt$ incident with both of them.
    \begin{itemize}
    \item When $p_1\in H\cup\LIS$ and $p_2\in H$, they define a great semicircle $\ell$ which satisfies the axiom.
    \item When $p_1,p_2\in\LIS$, the great semicircle $\LIS$ is incident with both $p_1$ and $p_2$.
    \end{itemize}
    \item For any two distinct lines $\ell_1$ and $\ell_2$, there is a unique point $p$ incident with both of them.
    \begin{itemize}
    \item When $\ell_1\neq\LIS$ and $\ell_2\neq\LIS$, then because they are both great semicircles, there must exist a point $p$ incident with both $\ell_1$ and $\ell_2$.
    \item When $\ell_1\neq\LIS$ and $\ell_2=\LIS$, then there exists a point $p=(\rho,\theta,\pi/2)$ which is incident with both $\ell_1$ and $\ell_2$.
    \end{itemize}
    \item There exist at least four unique points such that no line is incident with more than two of them.
    \vspace{1mm}\\
    The points $(\rho,0,\pi/4)$, $(\rho,\pi/2,\pi/4)$, $(\rho,\pi,\pi/4)$, and $(\rho,-\pi/2,\pi/4)$ satisfy this axiom.\qedhere
    \end{enumerate}
\end{proof}

\begin{definition}\label{def:VPn}
    % Given a projective space $\descRt$, let its \textit{vector space construction} $\descVt$ be the quotient set  \mbox{$\bbR^{3}\backslash\{\mathbf{0}\}$}
    Let the \textit{vector space construction} $\descVt$ of the real projective plane be the $\bbR^3$ vector space. Let the line at infinity be $\LIV=\{(x,y,0) \mid  (x,y)\in\bbR^{2}\}$
    Points in $\descVt$ are 1-dimensional subspaces such that a point $p=\text{span}(\mathbf{v}$) where $\mathbf{v}\in\bbR^3\backslash\{\mathbf{0}\}$.
    % (1-dimensional subspaces)  by the equivalence relation  \mbox{$a\sim b\mid a=\lambda b,\lambda\in\bbR\backslash\{0\}$}.
    Likewise, lines in $\descVt$ are 2-dimensional subspaces such that a line $\ell=\Span(\mathbf{v},\mathbf{u})$ where $\mathbf{v},\mathbf{u}\in\bbR^3\backslash\{\mathbf{0}\}$ and $\mathbf{v}\neq\mathbf{u}$.
\end{definition}

\begin{theorem}
    The vector space construction $(\descVt)$ is a construction of the real projective plane.
\end{theorem}
\begin{proof}
    We can show this meets the axioms of projective geometry as follows:
    \begin{enumerate}
    \item For any two distinct points $p_1=\Span(\langle p_{1x},p_{1y},p_{1z}\rangle)$ and $p_2=\Span(\langle p_{2x},p_{2y},p_{2z}\rangle)$, there exists a unique line $\ell\subset\descVt$ incident with both of them.
    \vspace{1mm}\\
    The span of the two points satisfy the axiom, such that $\ell=\Span(p_1,p_2)$.
    \item For any two distinct lines $\ell_1$ and $\ell_2$, there is a unique point $p$ incident with both of them.
    \vspace{1mm}\\
    For any two distinct 2-dimensional subspaces, there exists a unique 1-dimensional subspace incident with both of them.
    \item There exist at least four unique points such that no line is incident with more than two of them.
    \vspace{1mm}\\
    The points $\Span(\langle0,0,1\rangle)$, $\Span(\langle1,0,1\rangle)$, $\Span(\langle0,1,1\rangle)$, and $\Span(\langle1,1,1\rangle)$ satisfy this axiom.\qedhere
    \end{enumerate}
\end{proof}
%%%%%%%%%%%%%%%%%%%%%%%%%%%%%%%%%%%%%%%%%%%%%%%%%%%%%%%%%%%%%%%%%%%%%%%%%%%
%%%%%%%%%%%%%%%%%%%%%%%%%%%%%%%%%%%%%%%%%%%%%%%%%%%%%%%%%%%%%%%%%%%%%%%%%%%
%%%%%%%%%%%%%%%%%%%%%%%%%%%%%%%%%%%%%%%%%%%%%%%%%%%%%%%%%%%%%%%%%%%%%%%%%%%

%%%%%%%%%%%%%%%%%%%%%%%%%%%%%%%%%%%%%%%%%%%%%%%%%%%%%%%%%%%%%%%%%%%%%%%%%%%
%%%%%%%%%%%%%%%%%%%%%%%%%%%%%%% Section %%%%%%%%%%%%%%%%%%%%%%%%%%%%%%%%%%%
%%%%%%%%%%%%%%%%%%%%%%%%%%%%%%%%%%%%%%%%%%%%%%%%%%%%%%%%%%%%%%%%%%%%%%%%%%%
\section{$\iso_{PS}:\descR^2_P\to \descR^2_S$}\label{sec:IPS}
%%%%%%%%%%%%%%%%%%%%%%%%%%%%%%%%%%%%%%%%%%%%%%%%%%%%%%%%%%%%%%%%%%%%%%%%%%%
The isomorphism $\isoPS:\descPt\to\descSt$ must map points from $\descPt$ to points in $\descSt$ and lines in $\descPt$ to lines in $\descSt$, the latter being great semicircles.
To find $\isoPS$, consider a point 
\(p = (p_1, p_2)\in\bbR^2\subset\descPt\)
with $p_1\neq0$
in Cartesian coordinates.
We can express this point in polar coordinates as 
\(p = (r, \alpha) = (\sqrt{p_1^2 + p_2^2}, \tan^{-1}p_2/p_1)\). 
Similarly, let 
\(s = (s_1, s_2, s_3)\) 
be a point in \(\descSt\), described in spherical coordinates as 
\begin{equation*}
s = (\rho, \theta, \phi) = \left(\sqrt{s_1^2 + s_2^2 + s_3^2}, \tan^{-1}\frac{s_2}{s_1}, \tan^{-1}\frac{\sqrt{s_1^2 + s_2^2}}{s_3}\right). 
\end{equation*}
In \(\descSt\), \(\rho\) remains constant reducing \(s\) to two independent variables \(\theta\) and \(\phi\).

We consider the case where 
\(\alpha = \theta\)
and \(\phi\) is a strictly increasing function of \(r\) with
\(\phi(0) = 0\). 
By investigating the system while holding \(r\) or \(\alpha\) constant, we can gain insights into the behavior of \(\isoPS\).

A set of points in $\descPt$ with constant \(\alpha\)  corresponds to a line passing through the origin. 
In \(\descSt\), this should result in a great semicircle passing through the peak with constant \(\theta\) as seen in Figure~\ref{fig:enter-label}. 
On the other hand, a set of points in $\descPt$ with constant $r$ corresponds to a circle around the origin, and in \(\descSt\), it becomes a full circle around the peak with constant \(\phi\) as seen in Figure~\ref{fig:enter-label}. 
The radius of the circle changes with different \(r\), and similarly, in \(\descSt\), the radius changes as \(r\) varies, with larger \(r\) corresponding to larger \(\phi\). 
Notably, as \(r\to\infty\), \(\phi\to\frac{\pi}{2}\) such that if $p\in\LIP$ then $\isoPS(p)=(\rho,\theta,\pi/2)$ as proven later in Remark~\ref{remark:cont}.
This is also shown through the relation $\descPt\cong\lim_{r\to\infty}r\cdot\operatorname{proj}_{\descSt}{\bbR^2}$ as seen in Figure~\ref{fig:Projection_m}.

\begin{figure}[t]
    \centering
    \includegraphics[width=\textwidth]{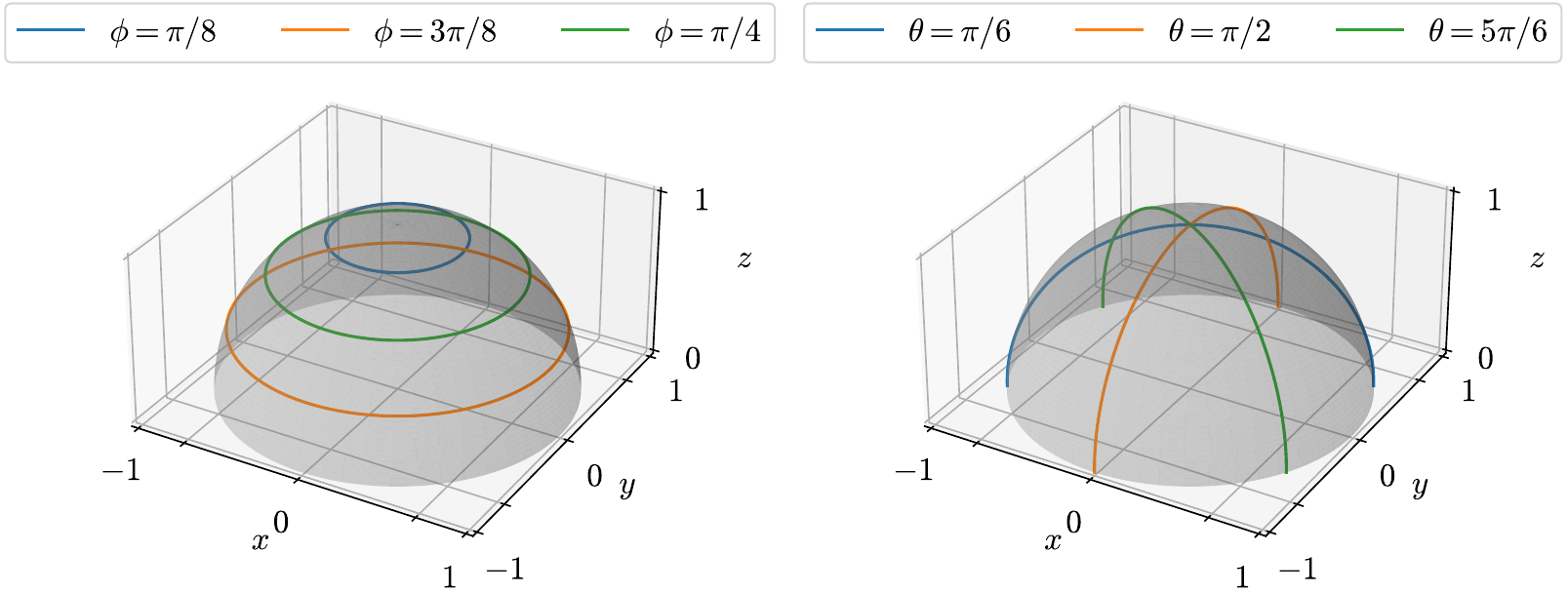}
    \caption{
    \textbf{Left:} Visualization of $s=\isoPS(p)$ for all $p=(\alpha, r)\in\descPt$ with constant $r$ corresponding to constant $\phi$ as shown in the legend.
    \textbf{Right:} Visualization of $s=\isoPS(p)$ for all $p=(\alpha, r)\in\descPt$ with constant $\alpha$ corresponding to constant $\theta$ as shown in the legend.
    }
    \label{fig:enter-label}
\end{figure}

Using the assumptions discussed above, we know that in \(\descPt\), lines of the form \(y = mx + b\) become great semicircles in \(\descSt\) and depending on the value of \(b\), the semicircle tilts from the peak $(0,0,1)$ onto the \(xy\)-plane.
We can use this behavior to find the equation for $\isoPS$.
Specifically, because great semicircles are centered at the origin, we can find them with a normal vector.
Using our knowledge about how $\isoPS$ needs to behave, we can find the plane which contains the great semicircle using the normal vector to that plane.
This normal vector has the form $\nvec =\langle m,-1,b\rangle$.
% For a more generalized approach one can add variables $c$ and $a$ such that the normal vector becomes $\nvec =\langle m,-1,cb\rangle$.
Then, we can find the points $s$ on the plane using $0=\nvec \cdot s = mx -y + bz = m\rho\sin\phi\cos\theta - \rho\sin\phi\sin\theta + b\rho\cos\phi$.
Because we already defined $\theta$ and $\rho$, we only need to solve for $\phi$:
\begin{align*}
    &0 = m\rho\sin\phi\cos\theta - \rho\sin\phi\sin\theta + b\rho\cos\phi \\
    &\Rightarrow  \tan\phi = \frac{-b}{m\cos\theta - \sin\theta} = \frac{y - mx}{\frac{y}{\sqrt{x^2 + y^2}} - \frac{mx}{\sqrt{x^2 + y^2}}} = \sqrt{x^2 + y^2}\\
    &\Rightarrow \phi = \tan^{-1} \sqrt{x^2+y^2}.
\end{align*}
% This can be done using the substitutions $b=y-mx$, $\cos\theta=x/\sqrt{x^2+y^2}$, and $\sin\theta=y/\sqrt{x^2+y^2}$ resulting in $\phi=\tan^{-1}\sqrt{x^2+y^2}$.

\begin{figure}[t]
    \centering
    \includegraphics[width=\linewidth]{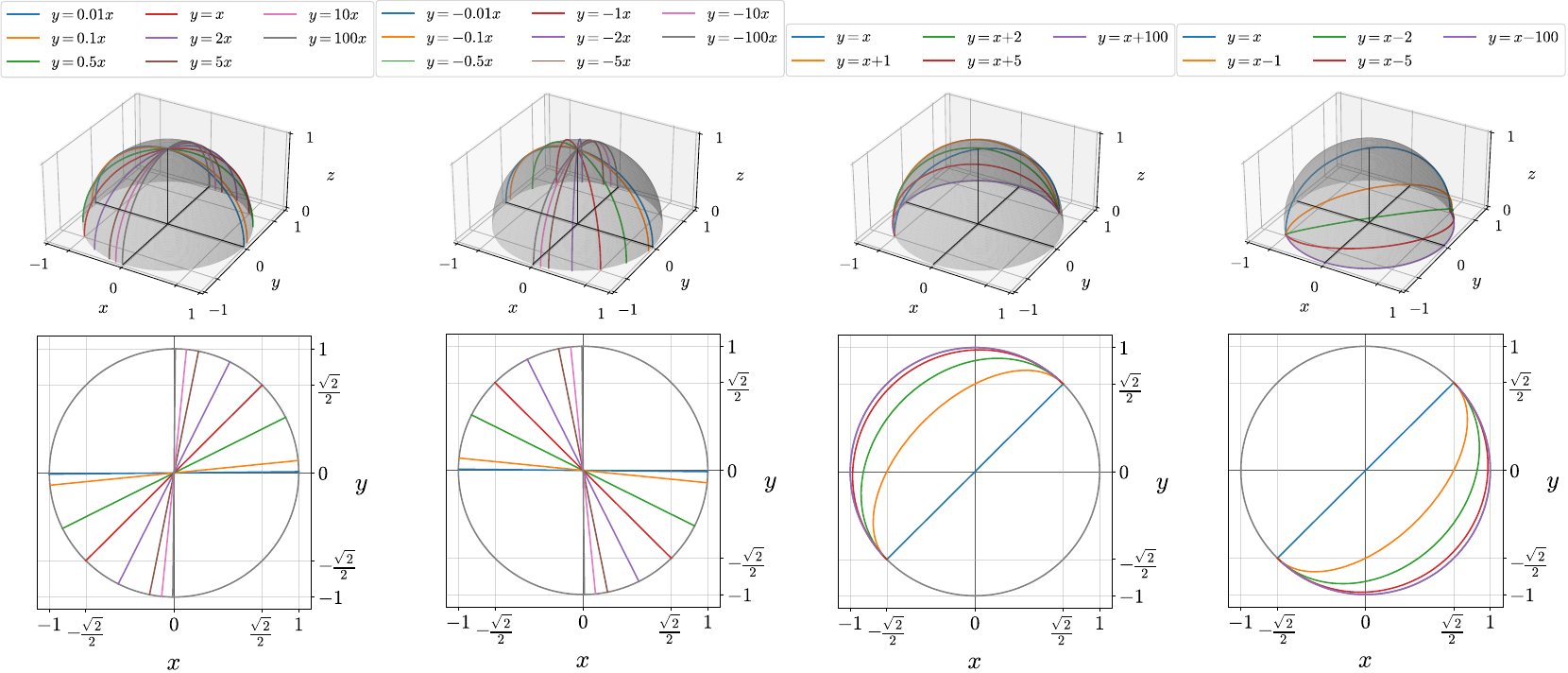}
    \caption{
    Visualization of $\isoPS$ through the relation 
    \mbox{$\descPt\cong\lim_{r\to\infty}r\cdot\operatorname{proj}_{\descSt}{\mathbb{R}^2}$}
    or 
    \mbox{$\descPt\cong\lim_{\rho\to\infty}\operatorname{proj}_{\descSt(\rho)}{\mathbb{R}^2}$}.
    The top legends contain equations of lines 
    \mbox{$\ell_1,\ldots,\ell_m\subset\descP^2$};
    the top plots show 
    \mbox{$\isoPS(\ell)\subset\descSt$};
    the bottom plots show
    \mbox{$\operatorname{proj}_{\descSt}{\mathbb{R}^2}$}.
    }
    \label{fig:Projection_m}
\end{figure}

\begin{definition}
    The isomorphism $\isoPS:\descPt\to\descSt$ can be written as mapping Cartesian to spherical coordinates such that $\isoPS((x,y)) = (\rho,\theta,\phi)$ for all $(x,y)\in\descPt\backslash\LI$ where $\rho\in\bbR\mid \rho>0$, 
    \begin{equation*}
    \theta = \begin{cases} 
          \tan^{-1}y/x & \text{if } x > 0 \\
          \tan^{-1}y/x + \pi & \text{if } x < 0 \\
          \pi/2 & \text{if } x = 0,\\
       \end{cases}
    \end{equation*}
    and $\phi = \tan^{-1}\sqrt{x^2 + y^2}$.
    This entire expression can be simplified by looking at mapping polar to spherical coordinates such that $\isoPS((r,\alpha)) = (\rho, \alpha,\tan^{-1}r)$ for all $(r,\alpha)\in\descPt\backslash\LIP$.
    For the points at infinity $(m)\in\LIP\backslash\{(\infty)\}$, $\isoPS((m))=(\rho,\tan^{-1}r,\pi/2)$.
    Lastly, for the point of infinity $(\infty)$, $\isoPS((\infty))=(\rho,\pi/2,\pi/2)$.
\end{definition}

\begin{lemma}\label{lemma:ndots}
    % maps lines from P to S (not infty)
    Let $\ell$ be a non-vertical line in $\descPt$.
    Then, for all points $p\in\ell\backslash\LIP$,  $\isoPS(p)=s\in\descSt$ where $s$ is on a plane normal to $\nvec $ which is defined by $\ell$.
\end{lemma}
\begin{proof}
    This is true by design; $\isoPS$ is defined such that $0 = \nvec \cdot s$ is always true for any $s=\isoPS(p)$ where $p=(p_1,p_2)\in\ell\subset\descPt$ and for $\nvec $ defined by $\ell$ such that $\nvec =\langle m, -1, b\rangle$ when $p_1\neq0$ and $\nvec =\langle 1, 0, 0\rangle$ when $p_1=0$.
\end{proof}

\begin{remark}\label{remark:cont}
    $\isoPS(p)=s$ is continuous for $p=(p_x,p_y)$ as $p_x\to\pm\infty$. 
\end{remark}
\begin{proof}
    Given a line $\ell=\{x,mx+b\mid x\in\bbR\}\cup\{(m)\}\subset\descPt$, where $\lim_{x\to\infty}(x,mx+b)=p\in\ell$, 
    $\isoPS(p)=(\rho,\theta,\phi)$ where
    $\theta = \lim_{x\to\infty}\tan^{-1}\frac{mx+b}{x}=\tan^{-1}m$ and
    $\phi = \lim_{x\to\infty}\tan^{-1}\sqrt{x^2 + (mx+b)^2} = \pi/2$.
    Thus, $\isoPS(p)=\isoPS((m))$.
\end{proof}

\begin{theorem}\label{theorem:isoPS}
$\isoPS: \descPt\to\descSt$ is an isomorphism, or equivalently $\isoPS$ preserves line structure such that all points on a line $\ell=\{p_1,p_2,\dots\}\subset\descPt$. Then, when $\ell$ is mapped to $\descSt$ such that $\ell'=\isoPS(\ell)$, $\ell'$ must also equal $\{\isoPS(p_1),\isoPS(p_2),\dots\}\subset\descSt$.
\end{theorem}
\begin{proof}
For all points $p \in \ell, \isoPS(p)=s$ which, following from Lemma~\ref{lemma:ndots}, is incident with the plane normal to $\nvec $ which defines $\ell'\subset\descSt$.
Thus, $\isoPS$ preserves line structure and is an isomorphism.
\end{proof}
% Then, using $\alpha=\theta=\tan^{-1} y/x$, we can solve for $z$ and $\phi$ as 
%%%%%%%%%%%%%%%%%%%%%%%%%%%%%%%%%%%%%%%%%%%%%%%%%%%%%%%%%%%%%%%%%%%%%%%%%%%
%%%%%%%%%%%%%%%%%%%%%%%%%%%%%%%%%%%%%%%%%%%%%%%%%%%%%%%%%%%%%%%%%%%%%%%%%%%
%%%%%%%%%%%%%%%%%%%%%%%%%%%%%%%%%%%%%%%%%%%%%%%%%%%%%%%%%%%%%%%%%%%%%%%%%%%

%%%%%%%%%%%%%%%%%%%%%%%%%%%%%%%%%%%%%%%%%%%%%%%%%%%%%%%%%%%%%%%%%%%%%%%%%%%
%%%%%%%%%%%%%%%%%%%%%%%%%%%%%%% Section %%%%%%%%%%%%%%%%%%%%%%%%%%%%%%%%%%%
%%%%%%%%%%%%%%%%%%%%%%%%%%%%%%%%%%%%%%%%%%%%%%%%%%%%%%%%%%%%%%%%%%%%%%%%%%%
\section{$\iso_{PV}:\descR^2_P\to \descR^2_V$}\label{sec:IPV}
%%%%%%%%%%%%%%%%%%%%%%%%%%%%%%%%%%%%%%%%%%%%%%%%%%%%%%%%%%%%%%%%%%%%%%%%%%%
The isomorphism $\isoPV:\descPt\to\descVt$ maps points in $\descPt$ to points (1-dimensional subspaces) in $\descVt$ and lines in $\descPt$ to lines (2-dimensional subspaces) in $\descVt$.
To find $\isoPV$, we use much of the intuition derived above in Section~\ref{sec:IPS}.
We start by using the same normal vector $\nvec =\langle m, -1, b \rangle$.
Then, substituting $m=(y-b)/x$ and solving for $z$, we find $z=1$, such that for all  points $p=(p_x,p_y)\in\descPt\backslash\LIP$,
$\isoPV(p)=\Span(\langle p_x,p_y,1\rangle)$.
Likewise, using the same logic as before, it can be easily shown that for a point $(m)\in\LIP$, $\isoPV((m))=\Span(\langle 1,m,0\rangle)$ and $\isoPV((\infty))=\Span(\langle 0,1,0\rangle)$.

\begin{definition}
    The isomorphism $\isoPV:\descPt\to\descVt$ can be written as mapping Cartesian points to 1-dimensional subspaces such that $\isoPS((x,y)) = \Span(\langle x,y,1\rangle)$ for all $(x,y)\in\descPt\backslash\LI$.
    For the points at infinity $(m)\in\LIP\backslash\{(\infty)\}$, $\isoPS((m))=\Span(\langle1,m,0\rangle)$.
    Lastly, for the point of infinity $(\infty)$, $\isoPS((\infty))=\Span(\langle 0,1,0\rangle)$.
\end{definition}

\begin{lemma}\label{lemma:SV}
    % For any point $p\in\descPt$, $\isoPV(p) = \Span(\isoPS(p))$.\\\\
    % \noindent Proof. 
    % \begin{enumerate}[label={Case \arabic*:}, topsep=-2pt, itemsep=-2pt,leftmargin=*]
    % \item For a point $p=(p_x,p_y)\in\descPt\backslash\LI$, 
    % $\isoPS(p)=(\rho,\theta,\phi)$.
    % Translating $(\rho,\theta,\phi)$ to a Cartesian $(x,y,z)$, we find
    % $x=\rho\sin\phi\cos\theta$ and $y=\rho\sin\phi\sin\theta$.
    % To find $z$, we know
    % $\phi=\tan^{-1}\sqrt{x^2+y^2} =\tan^{-1}\frac{b}{\sin\theta - m\cos\theta}$.
    % Then $0=m\cos\theta-\sin\theta+b\cot\phi$ or $0=mx-y+bz$.
    % Substituting $m=(y-b)/x$, we find $z=1$. 
    % \item For a point $p=(m)\in\LI\backslash\{(\infty)\}$, $\isoPS(p)=(\rho,\tan^{-1}m,\pi/2)$. 
    % Translating into Cartesian coordinates, $(x,y,z)=(\rho\cos\theta,\rho\sin\theta,0)$ which is in $\isoPV(p)=\Span(\langle1,m,0\rangle)$.
    % \item For a point $p=(\infty)$, $\isoPS(p)=(\rho,\pi/2,\pi/2)$. 
    % Translating into Cartesian coordinates, $(x,y,z)=(0,1,0)$ which is in $\isoPV(p)=\Span(\langle0,1,0\rangle)$.
    % \end{enumerate}
    $\descSt$ is isomorphic with $\descVt$ such that for any point $p\in\descPt$, $\isoPV = \Span(\isoPS(p))$.
\end{lemma}
\begin{proof}
Both $\isoPS$ and $\isoPV$ were derived using the normal vector $\nvec =\langle m, -1, b\rangle$ such that $0=\nvec \cdot s$ is true for $\isoPS$ and $\isoPV$ with $s\in\descSt$ and $s\in\descVt$ respectively. It follows from Lemma~\ref{lemma:ndots} that $\isoPV = \Span(\isoPS(p))$.
\end{proof}

\begin{theorem}
$\isoPV: \descPt\to\descVt$ is an isomorphism.
\end{theorem}
\begin{proof}
\noindent It follows directly from Theorem~\ref{theorem:isoPS} and Lemma~\ref{lemma:SV} that $\isoPV$ is an isomorphism.
\end{proof}
%%%%%%%%%%%%%%%%%%%%%%%%%%%%%%%%%%%%%%%%%%%%%%%%%%%%%%%%%%%%%%%%%%%%%%%%%%%
%%%%%%%%%%%%%%%%%%%%%%%%%%%%%%%%%%%%%%%%%%%%%%%%%%%%%%%%%%%%%%%%%%%%%%%%%%%
%%%%%%%%%%%%%%%%%%%%%%%%%%%%%%%%%%%%%%%%%%%%%%%%%%%%%%%%%%%%%%%%%%%%%%%%%%%

%%%%%%%%%%%%%%%%%%%%%%%%%%%%%%%%%%%%%%%%%%%%%%%%%%%%%%%%%%%%%%%%%%%%%%%%%%%
%%%%%%%%%%%%%%%%%%%%%%%%%%%%%%% Section %%%%%%%%%%%%%%%%%%%%%%%%%%%%%%%%%%%
%%%%%%%%%%%%%%%%%%%%%%%%%%%%%%%%%%%%%%%%%%%%%%%%%%%%%%%%%%%%%%%%%%%%%%%%%%%
\section{Conclusion}\label{sec:Conclusion}
%%%%%%%%%%%%%%%%%%%%%%%%%%%%%%%%%%%%%%%%%%%%%%%%%%%%%%%%%%%%%%%%%%%%%%%%%%%
% In this paper, we have explored isomorphisms between different constructions of the real projective plane and their practical applications, particularly in the context of direction-sensitive photosensors. 
% These isomorphisms provide a bridge between various representations of the real projective plane.
% Because $\descPt$ $(=\bbR^2\cup\LI)$ is a generalized Eucildean plane, the isomorphisms discussed in this work can also be used for points in $\bbR^2$ itself.
% Additionally, we have examined the use of direction-sensitive photosensors in particle physics research, highlighting their potential to enhance event analysis by accurately capturing the direction of incident light. 
% This innovative technology leverages principles from projective geometry, offering valuable insights into its applications in experimental physics.

In this paper, we have explored the isomorphisms between different constructions of the real projective plane ($\descR^2$) and their applications in the context of direction-sensitive photosensors. 
We established that the extended Euclidean plane ($\descPt$), the hemisphere construction ($\descSt$), and the vector space construction ($\descVt$) are all isomorphic to the real projective plane, $\descR^2$. 
These isomorphisms allow us to map between these different constructions, providing valuable insights into their geometric relationships.
Because $\descPt$ $(=\bbR^2\cup\LIP)$ is a generalized Eucildean plane, the isomorphisms discussed in this work can also be used for points in $\bbR^2$ itself.

These findings have practical implications, especially in the context of direction-sensitive photosensors, where lenses can be used to transform the direction of incoming light into positions on a local plane. This transformation aligns with the concept of isomorphisms between projective plane constructions, providing a physical representation of these mathematical relationships.
%%%%%%%%%%%%%%%%%%%%%%%%%%%%%%%%%%%%%%%%%%%%%%%%%%%%%%%%%%%%%%%%%%%%%%%%%%%
%%%%%%%%%%%%%%%%%%%%%%%%%%%%%%%%%%%%%%%%%%%%%%%%%%%%%%%%%%%%%%%%%%%%%%%%%%%
%%%%%%%%%%%%%%%%%%%%%%%%%%%%%%%%%%%%%%%%%%%%%%%%%%%%%%%%%%%%%%%%%%%%%%%%%%%

%%% -------------------------------------------------------------------
%%% -------------------------------------------------------------------
%%% This is where we create the bibliography.

\end{document}